%% file: Least_weakly_compact_can_be_weakly_measurable.tex
\documentclass[12pt]{amsart}

\title[The least weakly compact can be unfoldable and more]{The least weakly compact cardinal can be unfoldable, weakly measurable and nearly $\theta$-supercompact}

\author[Cody]{Brent Cody}
 \address{B. Cody, The Fields Institute for Research in Mathematical Sciences, 222 College Street, Toronto, Ontario M5T 3J1, Canada}
 \email{bcody@fields.utoronto.ca}
 \urladdr{http://www.fields.utoronto.ca/~bcody/}

\author[Gitik]{Moti Gitik}
 \address{M. Gitik, Department of Mathematics, Tel Aviv University, 69978 Tel Aviv, Israel}
 \email{gitik@post.tau.ac.il}
 \urladdr{http://www.math.tau.ac.il/~gitik}

\author[Hamkins]{Joel David Hamkins}
 \address{J. D. Hamkins, Mathematics, The Graduate Center of The City University of New York,
         365 Fifth Avenue, New York, NY 10016 \&
         Mathematics, College of Staten Island of CUNY, Staten Island, NY 10314}
 \email{jhamkins@gc.cuny.edu}
 \urladdr{http://jdh.hamkins.org}
 \thanks{The research of the second author was partially supported by ISF Grant 234/08. The research of the third author has been supported in part by NSF grant DMS-0800762, PSC-CUNY grant 64732-00-42 and Simons Foundation grant 209252. Commentary concerning this paper can be made at \url{http://jdh.hamkins.org/least-weakly-compact}.}

\author[Schanker]{Jason A. Schanker}
 \address{J. A. Schanker, Mathematics and Computer Science, Manhattanville College, 2900 Purchase St., Purchase, NY 10577}
 \email{jason.schanker@mville.edu}

\usepackage{amssymb,bbm}
\include{MathMacrosJDH}
\renewcommand{\S}{\mathbb{S}}

\newcommand{\tail}{\mathrm{tail}}
\begin{document}

\begin{abstract}
 We prove from suitable large cardinal hypotheses that the least weakly compact cardinal can be unfoldable, weakly measurable and even nearly $\theta$-supercompact, for any desired $\theta$. In addition, we prove several global results showing how the entire class of weakly compact cardinals, a proper class, can be made to coincide with the class of unfoldable cardinals, with the class of weakly measurable cardinals or with the class of nearly $\theta_\kappa$-supercompact cardinals $\kappa$, for nearly any desired function $\kappa\mapsto\theta_\kappa$. These results answer several questions that had been open in the literature and extend to these large cardinals the identity-crises phenomenon, first identified by Magidor with the strongly compact cardinals.
\end{abstract}

\maketitle

\section{Introduction}

In this article, we shall prove that the least weakly compact cardinal can exhibit any of several much stronger large cardinal properties. Namely, the least weakly compact cardinal can be unfoldable, weakly measurable and nearly $\theta$-supercompact for any desired $\theta$.

\begin{maintheorem} Assuming a suitable large cardinal hypothesis, the least weakly compact cardinal can be unfoldable, weakly measurable and even nearly $\theta$-supercompact, for any desired $\theta$.
\end{maintheorem}

Meanwhile, the least weakly compact cardinal can never exhibit these extra large cardinal properties in $L$, and indeed, the existence of a weakly measurable cardinal in the constructible universe is impossible. Furthermore, in each case the extra properties are strictly stronger than weak compactness in consistency strength.

The detailed results appear in theorems \ref{Theorem.LeastWC=LeastUnfoldable}, \ref{Theorem.LeastWC=LeastWM}, \ref{Theorem.LeastWC=LeastUnf=LeastWM}, \ref{Theorem.LeastWCNearlySC} and \ref{Theorem.LeastWC=LeastWM=LeastUnf=LeastNearlyThetaSC}. After this, the more global results of section \ref{Section.EveryWC} show that the entire class of weakly compact cardinals can be made to coincide with the class of unfoldable cardinals, with the class of weakly measurable cardinals, and with the class of nearly $\theta_\kappa$-supercompact cardinals $\kappa$, with enormous flexibility in the map $\kappa\mapsto\theta_\kappa$.

Our results therefore extend the `identity-crises' phenomenon---first identified (and named) by Magidor \cite{Magidor76:HowLargeIsFirstStrC}---which occurs when a given large cardinal property can be made in various models to coincide either with much stronger or with much weaker large cardinal notions. Magidor had proved that the least strongly compact cardinal can be the least supercompact cardinal in one model of set theory and the least measurable cardinal in another. Here, we extend the phenomenon to weak measurability, partial near supercompactness and unfoldability. Specifically, the least weakly measurable cardinal coincides with the least measurable cardinal under the \GCH, but it is the least weakly compact cardinal in our main theorem. Similarly, the least cardinal $\kappa$ that is nearly $\kappa^{+}$-supercompact is measurable with nontrivial Mitchell order under the \GCH, but it is the least weakly compact cardinal here (and similar remarks apply to near $\kappa^{++}$-supercompactness and so on). The least unfoldable cardinal is strongly unfoldable in $L$, and therefore a $\Sigma_2$-reflecting limit of weakly compact cardinals there, but it is the least weakly compact cardinal in our main theorem. The global results of section \ref{Section.EveryWC} show just how malleable these notions are.

\section{How to destroy and resurrect weak compactness}\label{Section.DestroyingResurrectingWeakCompactness}

In order to produce a model for the main theorem in which the least weakly compact cardinal exhibits much stronger properties, our strategy will be to begin with a large cardinal $\kappa$ exhibiting those stronger properties already in the ground model and then undertake an iteration that kills off all of the weakly compact cardinals below $\kappa$ in part by adding a $\gamma$-Souslin tree for every such $\gamma$. In order to know that the desired strong property of $\kappa$ itself survives, we will need to mollify the effects of adding such a tree at stage $\kappa$ itself. For this, we shall rely on a key technical observation, due originally to Kunen \cite{Kunen78:SaturatedIdeals} and explained here in lemma \ref{Lemma.Souslin+Branch=Cohen}, that the forcing to create a homogeneous $\kappa$-Souslin tree, followed by forcing with that tree, is forcing equivalent merely to adding a Cohen subset of $\kappa$, a comparatively mild forcing. The proof of our main theorem will adapt this observation in order to show how we may kill off all the weakly compact cardinals below $\kappa$ while preserving the unfoldability, the weak measurability and the near $\theta$-supercompactness of $\kappa$.

So let us explain Kunen's observation that adding a homogeneous $\kappa$-Souslin tree and then forcing with that tree is forcing equivalent to adding a Cohen subset to $\kappa$. An excellent detailed account of this argument can be found in \cite{GitmanWelch2011:Ramsey-likeCardinalsII}. One of the main imports of this observation, expressed by theorem \ref{Theorem.NonWeaklyCompactBecomesLarge} below, is that if one should start with a suitably prepared large cardinal $\kappa$, such as a measurable or even a supercompact cardinal, then the forcing to add a $\kappa$-Souslin tree will kill the large cardinal property, even the weak compactness of $\kappa$, but subsequent forcing with the tree will transform the forcing altogether into the comparatively mild Cohen forcing, which can preserve the large cardinal. In this way, the intermediate extension has a non-weakly compact cardinal that is made measurable or even supercompact (or more) by forcing with a $\kappa$-Souslin tree.

For the details, let $\kappa$ be any inaccessible cardinal, and consider the forcing $\S$ to create a homogeneous $\kappa$-Souslin tree. Conditions are pairs $(t,f)$, where $t\of 2^\ltkappa$ is a homogeneous normal $(\alpha+1)$-tree, $\alpha<\kappa$, and $f$ is an enumeration of the automorphism group $\Aut(t)$. A normal $(\alpha+1)$-tree is a subtree of $2^\ltkappa$ of height $\alpha+1$, meaning that it has nodes on level $\alpha$, but no nodes on any higher level, which satisfies the {\df normality} condition, that every node in $t$ extends to a node on any higher level. Such a tree is said to be {\df homogeneous} if its automorphism group acts transitively on each level, which means that for any two nodes $p,q$ on the same level of $t$, there is an automorphism $\pi:t\iso t$ such that $\pi(p)=q$. The order on $\S$ is that $(t,f)\leq (t',f')$ just in case $t$ is an end extension of $t'$, and furthermore the $\alpha^\th$ automorphism $\pi_\alpha'=f'(\alpha)$ of $t'$ listed by $f'$ extends to the $\alpha^\th$ automorphism $\pi_\alpha=f(\alpha)$ of $t$ listed by $f$. In other words, the trees are growing up towards $\kappa$, and furthermore the automorphisms of every level are preserved all the way up in a way that is recorded by the enumeration of the automorphism group. (In particular, it isn't just that every automorphism of the smaller tree extends to an automorphism of the taller tree, but rather each lower automorphism has a canonical extension to an automorphism of the higher tree, namely, the automorphism with the same index in the enumeration.)

Forcing with $\S$ clearly adds a $\kappa$-tree $T$, since any condition can be extended to any height, and it is not difficult to see that $T$ is a $\kappa$-Souslin tree. Namely, given any name $\dot A$ for a maximal antichain in $T$, one undertakes a bootstrap argument to build a condition $(t,f)\in \S$ that decides $\dot A\intersect t$ and forces it to be maximal there (the automorphisms add a small complication, but the essential idea from the $\omega_1$ case works out fine). Thus, $(t,f)$ forces that $\dot A$ is bounded in $T$, and so $T$ is $\kappa$-Souslin. Consider now the forcing $\S*\dot T$, which first adds the tree $T$, and then forces with $T$. Conditions in $\S*\dot T$ amount to $(t,f,\tau)$, where $\tau$ is an $\S$-name for an element of $\dot T$. By another bootstrap argument, one may find a dense set of such conditions with the further property that $(t,f)$ decides the value of $\tau$ in $t$, forcing it to be a particular node on the top level of $t$. Thus, the iteration $\S*\dot T$ is forcing equivalent to the forcing $\R$ having conditions $(t,f,p)$, where $(t,f)\in\S$ and $p$ is a node on the top level of $t$, ordered in the obvious way. The key observation now is that $\R$ is $\ltkappa$-closed in $V$, since the nodes $p$ on any descending sequence of conditions can be used to specify a node on the limit level and a corresponding branch, whose automorphic images under the automorphisms will combine to form a homogeneous normal tree, providing us with a condition below the sequence. Thus, since $\R$ is $\ltkappa$-closed and has size $\kappa$, it follows that $\R$ is forcing equivalent to Cohen forcing $\Add(\kappa,1)$. From this (and the fact that $\kappa$ is inaccessible), it follows that $\S$ is $\ltkappa$-distributive and preserves all cardinals and cofinalities. Alternatively, one can deduce this by observing that $\S$ is strategically $\ltkappa$-closed, since player two can get through limits by explicitly extending paths as the game is played. In summary, the argument establishes:

\begin{lemma}[Kunen]\label{Lemma.Souslin+Branch=Cohen}
 If $\kappa$ is inaccessible, then there is a strategically $\ltkappa$-closed notion of forcing $\S$ of size $\kappa$ such that:
 \begin{enumerate}
  \item Forcing with $\S$ adds a homogenous $\kappa$-Souslin tree $T$.
  \item The combined forcing $\S*\dot T$ is forcing equivalent to Cohen forcing $\Add(\kappa,1)$.
 \end{enumerate}
\end{lemma}

For a further detailed account of this argument, we refer the reader to \cite{GitmanWelch2011:Ramsey-likeCardinalsII} or to \cite{Kunen78:SaturatedIdeals}.

\begin{theorem}[Kunen]\label{Theorem.NonWeaklyCompactBecomesLarge}
 It is a relatively consistent with \ZFC\ that a cardinal $\kappa$ is not weakly compact, but becomes weakly compact and indeed much more (measurable, strong, strongly compact, supercompact) in a forcing extension $V[G]$ obtained by forcing with a certain $\kappa$-Souslin tree.
\end{theorem}

\begin{proof}
Let us illustrate with the case of measurability. Suppose that $\kappa$ is measurable in $V$ and furthermore, that the measurability of $\kappa$ is preserved by forcing over $V$ with the forcing $\Add(\kappa,1)$ to add a Cohen subset to $\kappa$. This can be achieved by suitable preparatory forcing, for example, by an Easton support $\kappa$-iteration of such forcing below $\kappa$, combined with forcing to ensure the \GCH\ at $\kappa$. Consider the forcing $\S$ of lemma \ref{Lemma.Souslin+Branch=Cohen}, which adds a $V$-generic $\kappa$-Souslin tree $T$. Since the existence of this tree shows that the tree property fails for $\kappa$ in $V[T]$, it follows that $\kappa$ is not weakly compact there. But let us now force with $T$ over $V[T]$, adding a $V[T]$-generic branch $b$. Lemma \ref{Lemma.Souslin+Branch=Cohen} shows that the combined forcing $T*b\of \S*\dot T$ is isomorphic to $\Add(\kappa,1)$, and so $\kappa$ remains measurable in $V[T*b]$. So the non-weakly compact cardinal $\kappa$ in $V[T]$ became measurable in $V[T][b]$ by forcing with $T$. A similar argument works with any other large cardinal that can be made indestructible by the forcing $\Add(\kappa,1)$, and this includes most of the commonly considered large cardinal concepts, such as measurable, strong, strongly compact, supercompact and so on. So it is relatively consistent from such hypotheses that a non-weakly compact cardinal $\kappa$ in $V[T]$ becomes correspondingly measurable, strong, strongly compact or supercompact in $V[T][b]$, after forcing with that tree.
\end{proof}

One may think of the $\kappa$-Souslin tree $T$ as the `last' Souslin tree of $V[T]$, for if we chop down this last tree, by forcing with it over $V[T]$, we find in the resulting extension that there are no $\kappa$-Souslin trees at all.

\section{The least weakly compact cardinal can be unfoldable}

We would like now to prove that the least weakly compact cardinal can be unfoldable. This settles a question that has been open since the time unfoldable cardinals were first introduced in \cite{Villaveces1998:ChainsOfEndElementaryExtensionsOfModelsOfSetTheory}, and it has been considered more recently by various researchers. We shall subsequently show more, in theorem \ref{Theorem.WC=Unf=W}, that the class of weakly compact cardinals can coincide with the class of unfoldable cardinals, even when there are a proper class of them.

Let us briefly recall some definitions. A cardinal $\kappa$ is {\df unfoldable} if for every ordinal $\theta$ and every transitive set $M$ of size $\kappa$ with $\kappa\in M$, there is a transitive set $N$ and an elementary embedding $j:M\to N$ with critical point $\kappa$ and $j(\kappa)>\theta$. The cardinal $\kappa$ is {\df strongly unfoldable} if for every $\theta$ and every such $M$ there is such an embedding with $V_\theta\of N$. Similarly, $\kappa$ is {\df superunfoldable} if for every $\theta$ and every such $M$ with $M^\ltkappa\of M$ there is such an embedding with $N^\theta\subseteq N$. As it allows for a simpler argument to be carried out, in what follows we make use of the fact that a cardinal $\kappa$ is strongly unfoldable if and only if it is superunfoldable \cite[Corollary 7]{DzamonjaHamkins2006:DiamondCanFail}. We refer the reader to the articles \cite{Hamkins2001:UnfoldableCardinals,DzamonjaHamkins2006:DiamondCanFail,Johnstone2008:StronglyUnfoldableCardinalsMadeIndestructible,HamkinsJohnstone2010:IndestructibleStrongUnfoldability} for numerous equivalent formulations of these concepts and other background information and results.

Our first instance of the main theorem is the following.

\begin{theorem}\label{Theorem.LeastWC=LeastUnfoldable}
 If there is a strongly unfoldable cardinal, then there is a forcing extension in which it is the least weakly compact cardinal and unfoldable.
\end{theorem}

\noindent Since every unfoldable cardinal is strongly unfoldable in $L$, the theorem implies that if there is an unfoldable cardinal, then in a forcing extension of $L$, it is the least weakly compact cardinal and unfoldable. The consistency strength of the hypothesis of the theorem, therefore, is optimal. Note that the least weakly compact cardinal can never be strongly unfoldable, as strongly unfoldable cardinals are $\Sigma_2$-reflecting.

\begin{proof}[Proof of thm \ref{Theorem.LeastWC=LeastUnfoldable}]
Suppose that $\kappa$ is strongly unfoldable. By forcing if necessary, we may assume that the \GCH\ holds. Let us assume also, for convenience, that there are no inaccessible cardinals above $\kappa$; we shall explain later how to omit this assumption. We define an Easton-support forcing iteration $\P=\langle (\P_\gamma,\dot{\Q}_\gamma)\mid\gamma\leq\kappa\rangle$ of length $\kappa+1$ as follows. The forcing $\dot\Q_\gamma$ at stage $\gamma$ is trivial, unless $\gamma$ is an inaccessible cardinal in $V[G_\gamma]$, in which case $\dot{\Q}_\gamma$ is the two-step iteration $\Add(\gamma,1)*\dot\S_\gamma$ that adds a Cohen subset to $\gamma$ over $V[G_\gamma]$ and then adds a homogeneous $\gamma$-Souslin tree via the forcing of lemma \ref{Lemma.Souslin+Branch=Cohen}. That is, $\dot\Q_\gamma$ is a $\P_\gamma$-name that is forced by conditions in $\P_\gamma$ either to be trivial or to be the two-step forcing we mentioned, to the extent that those conditions force $\gamma$ to be inaccessible in $V[\dot G_\gamma]$. At the top of the iteration, at stage $\kappa$, let $\dot{\Q}_\kappa=\dot{\Add}(\kappa,\kappa^+)$ be a $\P_\kappa$-name for the forcing to add $\kappa^+$ many Cohen subsets to $\kappa$ as defined in $V^{\P_\kappa}$. Suppose that $G*g\of\P_\kappa*\dot\Q_\kappa$ is $V$-generic, and consider the model $V[G][g]$. It is easy to see that there are no weakly compact cardinals below $\kappa$ in $V[G][g]$ as follows. For each inaccessible cardinal $\gamma<\kappa$ that remains inaccessible in the partial extension by $\P_\gamma$, the stage $\gamma$ forcing explicitly adds a $\gamma$-Souslin tree via $\S_\gamma$, and this tree survives through all the subsequent forcing after stage $\gamma$ because this later forcing is strategically $\leqgamma$-closed. These trees violate the tree property for $\gamma$, and so no such $\gamma$ below $\kappa$ is weakly compact in the final model $V[G][g]$.

It remains to verify that $\kappa$ remains unfoldable in $V[G][g]$. For this, suppose $A\of\kappa$ in $V[G][g]$ and consider any large ordinal $\theta$, which by increasing if necessary we may assume is a beth fixed point $\beth_\theta=\theta$. Since the stage $\kappa$ forcing is $\kappa^+$-c.c., it follows that $A\in V[G][g\restrict\beta]$ for some $\beta<\kappa^+$. Since adding $\beta$ many Cohen subsets to $\kappa$ is isomorphic to adding just one, we may apply an automorphism of the forcing and assume without loss of generality that $A\in V[G][g_0]$, where $g_0\of\kappa$ is the Cohen set added by $g$ in the very first coordinate. Thus, $A=\dot A_{G*g_0}$ for some $\P_\kappa*\Add(\kappa,1)$-name $\dot A\in H_{\kappa^+}$. Since $\kappa$ is strongly unfoldable in $V$, the results of \cite{DzamonjaHamkins2006:DiamondCanFail} show that $\kappa$ is also superunfoldable there. In fact, by \cite[Lemma 5]{DzamonjaHamkins2006:DiamondCanFail} there is a transitive set $M\prec H_{\kappa^+}$ with $\dot A\in M$, $M^\ltkappa\of M$ and an elementary embedding $j:M\to N$, where $\theta<j(\kappa)$, $N^\theta\of N$ and $|N|^V=\beth_{\theta+1}=\theta^+$. The forcing $j(\P_\kappa)$ factors as $\P_\kappa*\dot{\Add}(\kappa,1)*\dot\S_\kappa*\dot{\P}_{\tail}$, where $\dot{\P}_{\tail}$ is a $j(\P)_{\kappa+1}$-name for the tail of the iteration $j(\P_\kappa)$ beyond stage $\kappa$. Let $g_1$ be the Cohen subset added on the next coordinate of $g$, so that $g_0$ and $g_1$ are mutually generic over $V[G]$. By lemma \ref{Lemma.Souslin+Branch=Cohen}, we may view $g_1$ as first adding a $\kappa$-Souslin tree $T$ and then forcing with this tree to add a branch $b$, so that $g_1=T*b\of\S_\kappa*\dot T$ where $\S_\kappa=(\dot{\S}_\kappa)_{G*g_0}$. In particular, $g_0*T$ is $N[G]$-generic for the stage $\kappa$ forcing arising in $j(\P_\kappa)$. Since $N^\theta\of N$ in $V$, it follows that $N[G][g_0][T]^\theta\of N[G][g_0][T]$ in $V[G][g_0][T]$ by the chain condition on $\P_\kappa*\dot{\Add}(\kappa,1)*\dot{\S}_\kappa$. Thus, because we assumed that there are no inaccessible cardinals greater than $\kappa$ in $V$, it follows that the next stage of forcing in $\Ptail=(\dot{\P}_{\tail})_{G*g_0*T}$ is beyond $\theta$, and so $\Ptail$ is strategically $\leqtheta$-closed in $N[G][g_0][T]$. Moreover, since $|N|^V=\theta^+$, it follows that in $V[G][g_0][T]$ we may construct by diagonalization an $N[G][g_0][T]$-generic filter $\Gtail\of\Ptail$. Thus, we may lift the embedding in $V[G][g_0][T]$ to $j:M[G]\to N[j(G)]$, where $j(G)=G*g_0*T*\Gtail$. Since $N[j(G)]^\theta\of N[j(G)]$ in $V[G][g_0][T]$, we may furthermore use $\bigcup j"g_0=\bigcup g_0$ as a master condition and again by diagonalization find an $N[j(G)]$-generic filter $j(g_0)\of\Add(j(\kappa),1)^{N[j(G)]}$. Thus, we may lift the embedding through this forcing to achieve $j:M[G][g_0]\to N[j(G)][j(g_0)]$. From $A=\dot A_{G*g_0}$ we know that $A\in M[G][g_0]$, and since $\theta<j(\kappa)$ it follows that our lifted embedding witnesses this instance of unfoldability in $V[G][g]$. Thus, $\kappa$ is unfoldable in $V[G][g]$, as desired.

Finally, we may omit the assumption that there are no inaccessible cardinals above $\kappa$ by first forcing to make the strong unfoldability of $\kappa$ indestructible by $\ltkappa$-closed $\kappa^+$-preserving forcing, using results of \cite{HamkinsJohnstone2010:IndestructibleStrongUnfoldability}, and then forcing above $\kappa$ so as to destroy every inaccessible cardinal, by using class forcing to create a class club that avoids the inaccessibles, and then forcing so as to collapse cardinals not in this club. Alternatively, if all that is desired is a transitive model in which the least weakly compact cardinal is unfoldable, then one may instead attain the assumption that there are no inaccessible cardinals above $\kappa$ simply by cutting the universe off at the least inaccessible cardinal above $\kappa$ and then working in that cut-off universe.
\end{proof}

Observe that in the above proof, the superunfoldability of $\kappa$ is used to ensure that the model $N$ knows that there are no inaccessible cardinals in the interval $(\kappa,\theta]$, and hence that $\P_{\tail}$ is strategically ${\leq}\theta$-closed in $N[G][g_0][T]$. Since the assumption that $\kappa$ is unfoldable does not provide such a situation, the mere unfoldability of $\kappa$ does not allow one to run the above argument. Although, a similar argument can be carried out from the seemingly weaker assumption on $\kappa$ that for every ordinal $\theta$ and every transitive $M \prec H_{\kappa^+}$ of size $\kappa$, there exists a $\theta$-unfoldability embedding $j: M \rightarrow N$ with critical point $\kappa$ having the additional property that $N$ thinks that there are no inaccessible cardinals in the interval $(\kappa,\theta]$.  By forcing less often, we could weaken the assumption further by replacing inaccessible with weakly compact in the above comments.

We note furthermore how our proof method, and in particular, our apparent need to destroy large cardinals above $\kappa$, appears to mesh neatly with theorem 3.6 in \cite{Villaveces1998:ChainsOfEndElementaryExtensionsOfModelsOfSetTheory}, asserting that if there is a Ramsey cardinal, then the least weakly compact is strictly less than the least unfoldable cardinal.

\section{The least weakly compact cardinal can be weakly measurable}

We next prove that the least weakly compact cardinal can be weakly measurable. This settles the main question left open in Schanker's work \cite{Schanker2011:WeaklyMeasurableCardinals,Schanker2011:Dissertation}, where the weakly measurable cardinals are introduced.

Recall that an inaccessible cardinal $\kappa$ is {\df weakly measurable} if for every transitive set $M$ of size $\kappa^+$ with $\kappa\in M$, there is a transitive set $N$ and an elementary embedding $j:M\to N$ with critical point $\kappa$. Equivalently, $\kappa$ is weakly measurable if for any collection $\mathcal{A}$ of $\kappa^+$ many subsets of $\kappa$, there is a $\kappa$-complete nonprincipal filter $F$ measuring every set in $\mathcal{A}$. In this way and others, the weakly measurable cardinals extend various weak compactness concepts from the realm of objects of size $\kappa$ up to objects of size $\kappa^+$, thereby forming a hybrid notion between weak compactness and measurability.  Under the \GCH, of course, every weakly measurable cardinal is fully measurable, and Schanker proved in any case that every weakly measurable cardinal is measurable in an inner model. So in this sense, the weakly measurable cardinals seem fairly close to the measurable cardinals. Nevertheless, Schanker established that the two concepts are not equivalent---any measurable cardinal can become non-measurable, while remaining weakly measurable, in a forcing extension---and he inquired how far weak measurability could deviate from measurability. He inquired specifically whether the least weakly compact cardinal can be weakly measurable. Here, we answer affirmatively.

\begin{theorem}\label{Theorem.LeastWC=LeastWM}
 If there is a measurable cardinal, then there is a forcing extension in which it is the least weakly compact cardinal and still weakly measurable.
\end{theorem}

\begin{proof}
Suppose $\kappa$ is a measurable cardinal in $V$. By forcing if necessary, we may also assume that $2^\kappa=\kappa^+$ in $V$, since the forcing to achieve this adds no subsets to $\kappa$ and consequently preserves the measurability of $\kappa$. We define a forcing iteration $\P=\langle (\P_\gamma,\dot{\Q}_\gamma)\mid\gamma\leq\kappa\rangle$ of length $\kappa+1$ using Easton support as follows. The forcing $\dot\Q_\gamma$ at stage $\gamma$ is trivial, unless $\gamma$ is an inaccessible cardinal in $V[G_\gamma]$, in which case $\dot{\Q}_\gamma$ names the two-step iteration $\Add(\gamma,\gamma^+)*\dot\S_\gamma$ in $V[G_\gamma]$ that adds a Cohen subset to $\gamma$ and then adds a homogeneous $\gamma$-Souslin tree via the forcing of lemma \ref{Lemma.Souslin+Branch=Cohen}. At the top of the iteration, at stage $\kappa$, let $\dot{\Q}_\kappa=\dot{\Add}(\kappa,\kappa^{++})$ be a $\P_\kappa$-name for the forcing to add $\kappa^{++}$ many Cohen subsets to $\kappa$ as defined in $V^{\P_\kappa}$. Suppose that $G*g\of\P_\kappa*\dot{\Q}_\kappa$ is $V$-generic for this forcing, and consider the model $V[G][g]$. As in the proof of theorem \ref{Theorem.LeastWC=LeastUnfoldable}, one can easily argue that there are no weakly compact cardinals less than $\kappa$ in $V[G][g]$.

It remains to show that $\kappa$ remains weakly measurable in $V[G][g]$. For this, fix any $A\of P(\kappa)^{V[G][g]}$ of size at most $\kappa^+$ in $V[G][g]$. Since the forcing $\Add(\kappa,\kappa^{++})$ at stage $\kappa$ is $\kappa^+$-c.c. in $V[G]$, it follows that $A\in V[G][g\restrict\delta]$ for some $\delta<\kappa^{++}$, where $g\restrict\delta=g\intersect\Add(\kappa,\delta)$ is the natural restriction of the forcing to add only the first $\delta$ many subsets of $\kappa$. Since $\delta$ has cardinality at most $\kappa^+$, we may apply an automorphism of the forcing and thereby assume without loss of generality that $\delta=\kappa^+$. Thus, $A=\Adot_{G*g\restrict\kappa^+}$ for some $\P_\kappa*\dot{\Add}(\kappa,\kappa^+)$-name $\Adot$, and we may assume furthermore that $\Adot\in H_{\kappa^{++}}^V$. Let $j:V\to M$ be the ultrapower by a normal measure on $\kappa$ in $V$, so $j$ is an elementary embedding of $V$ into a transitive class $M$ and $j$ has critical point $\kappa$. Our strategy will be to lift this embedding to the partial forcing extension $j:V[G][g\restrict\kappa^+]\to M[j(G)][j(g\restrict\kappa^+)]$, defining the lift in the full extension $V[G][g]$. If we can do this, then since $A=\Adot_{G*g\restrict\kappa^+}$, it will follow that $A$ is a subset of the domain, and therefore will be measured by the filter induced by this embedding, that is, by the filter $F$ generated by $\set{X\of\kappa\st X\in V[G][g\restrict\kappa^+]\text{ and }\kappa\in j(X)}$.  Notice that $F$ will be $\kappa$-complete in $V[G][g]$, since the domain is closed under $\ltkappa$-sequences in $V[G][g]$, that is, $V[G][g\restrict\kappa^+]^\ltkappa\of V[G][g\restrict\kappa^+]$ in $V[G][g]$ and $\cp(j)=\kappa$.  In this way, we will have verified this instance of weak measurability and thereby be able to conclude that $\kappa$ is weakly measurable in $V[G][g]$, as desired.

So let's lift the embedding. Specifically, consider first the forcing $j(\P_\kappa)$, which is isomorphic in $M$ to $\P_\kappa*\dot{\Add}(\kappa,\kappa^+)*\dot{\S}_\kappa*\dot{\P}_\tail$, where $\dot{\P}_{\tail}$ is a $\P*\dot{\Add}(\kappa,\kappa^+)*\dot{\S}_\kappa$-name for the tail of the iteration $j(\P_\kappa)$ beyond stage $\kappa$. We have an $M$-generic filter $G*(g\restrict\kappa^+)\of\P_\kappa*\dot{\Add}(\kappa,\kappa^+)$. Let $B$ be one of the Cohen sets added by $g$ beyond the part $g\restrict\kappa^+$ that we have used so far. By lemma \ref{Lemma.Souslin+Branch=Cohen}, we may view $B$ as $B=T*b\of\S_\kappa*\dot T$, where $T$ is a generic homogeneous $\kappa$-Souslin tree, mutually generic with the rest of $g$. Thus, $G*(g\restrict\kappa^+)*T$ is $M$-generic for the forcing of $j(\P_\kappa)$ up to and including stage $\kappa$. The remaining forcing $\Ptail=(\dot{\P}_{\tail})_{G*(g\restrict\kappa^+)*T}$ is strategically $\ltkappa^+$-closed in $M[G*(g\restrict\kappa^+)*T]$, a model which is closed under $\kappa$-sequences in $V[G*(g\restrict\kappa^+)*T]$ by virtue of $M$ exhibiting this closure in $V$ and $j(\P)_{\kappa + 1}$ being $\kappa^+$-c.c. Since $\P_\kappa$ has at most $2^{\kappa}$ many (dense) subsets, it follows by elementarity that $j(\P_\kappa)$ (and hence $j(\P)_{\kappa + 1}$) will have at most $|j(2^{\kappa})|^V = |j(\kappa^+)|^V=\kappa^+$ many (dense) subsets in $M$, as counted in $V$. We may therefore construct an $M[G*(g\restrict\kappa^+)*T]$-generic filter $\Gtail\of\Ptail$ by a diagonalization argument in $V[G*(g\restrict\kappa^+)*T]$. Thus, we may lift the embedding to $j:V[G]\to M[j(G)]$ in $V[G*(g\restrict \kappa^+)*T]$, where $j(G)=G*(g\restrict\kappa^+)*T*\Gtail$.

We now lift the embedding to $V[G][g\restrict\kappa^+]$ in $V[G][g]$. Notice that the partial order $j(\Add(\kappa,\kappa^+))=\Add(j(\kappa),j(\kappa^+))^{M[j(G)]}$ is ${<}j(\kappa)$-directed closed in $M[j(G)]$, a model that is closed under $\kappa$-sequences in $V[G*(g\restrict\kappa^+)*T]$. In lieu of a master condition argument we will construct a ``master filter'' by building a decreasing sequence of increasingly masterful conditions. For any $\beta<\kappa^+$, we have $j\image\beta\in M$ and consequently $j\image (g\restrict\beta)\in M[j(G)]$. Furthermore, if $a\of\Add(j(\kappa),j(\kappa^+))$ is a maximal antichain in $M[j(G)]$, then by the chain condition it follows that $|a|^{M[j(G)]}\leq j(\kappa)$ and consequently $a\of\Add(j(\kappa),\delta)$ for some $\delta<j(\kappa^+)$. Since $j$ is the ultrapower by a normal measure, it follows that $j(\kappa^+)=\sup j\image\kappa^+$, and so $a\of\Add(j(\kappa),j(\beta))$ for some $\beta<\kappa^+$. Suppose that we have a condition $q\in\Add(j(\kappa),j(\kappa^+))$ that is compatible with each element of $j\image (g\restrict\kappa^+)$. In particular, $q\restrict j(\beta)$ is compatible in $\Add(j(\kappa),j(\beta))$ with the master condition $\bigcup j\image (g\restrict\beta)\in \Add(j(\kappa),j(\beta))$, and so we may find an extension $q^+$ of $q$, adding to the domain of $q$ only below $j(\beta)$, such that $q^+$ meets $a$ and $q^+$ remains compatible with every element of $j\image (g\restrict\kappa^+)$. Since $|j(2^\kappa)|^V=|j(\kappa^+)|^V=\kappa^+$, we may enumerate in $V[G][g]$ all of the maximal antichains of $M[j(G)]$ in a $\kappa^+$-sequence, and then by iteratively applying the observation of the previous sentence, we build a continuous descending $\kappa^+$-sequence of conditions $q_\alpha\in\Add(j(\kappa),j(\kappa^+))$, such that every $q_\alpha$ is compatible with $j\image (g\restrict\kappa^+)$ and $q_{\alpha+1}$ meets the $\alpha^\th$ enumerated antichain.  Note that we are using the fact that $M[j(G)]$ is closed under $\kappa$-sequences in $V[G*(g\restrict\kappa^+)*T]$ and that $\Add(j(\kappa),j(\kappa^+))^{M[j(G)]}$ is ${<}\kappa$-directed closed to justify the construction at limit stages.  It follows that the filter $h$ generated by the $q_\alpha$'s is what we call a \emph{master filter}, meaning that it is $M[j(G)]$-generic and contains $j\image(g\restrict\kappa^+)$. Thus, $j$ lifts to $j:V[G][g\restrict\kappa^+]\to M[j(G)][j(g\restrict\kappa^+)]$, where $j(g\restrict\kappa^+)=h$. As we explained in the previous paragraph, it now follows, since $A\of V[G][g\restrict \kappa^+]$, that in $V[G][g]$ we have a $\kappa$-complete filter $F$ measuring every set in $A$, thereby witnessing this instance of weak measurability. So $\kappa$ is weakly measurable in $V[G][g]$, as desired.

\end{proof}

By combining the methods of theorems \ref{Theorem.LeastWC=LeastUnfoldable} and \ref{Theorem.LeastWC=LeastWM}, we
show that the least weakly compact cardinal can be simultaneously unfoldable and weakly measurable.

\begin{theorem}\label{Theorem.LeastWC=LeastUnf=LeastWM}
 If $\kappa$ is measurable and strongly unfoldable, then there is a forcing extension in which $\kappa$ is the least weakly compact cardinal, unfoldable and weakly measurable.
\end{theorem}

\begin{proof}
We combine the arguments of theorems \ref{Theorem.LeastWC=LeastUnfoldable} and \ref{Theorem.LeastWC=LeastWM}. Suppose that $\kappa$ measurable and strongly unfoldable. By forcing (or by chopping the universe off), if necessary, we may assume that there are no inaccessible cardinals above $\kappa$, and similarly by forcing we may assume that the \GCH\ holds. Let $\P$ be the length $\kappa+1$ Easton support iteration of theorem \ref{Theorem.LeastWC=LeastWM}, which forces with $\Add(\gamma,\gamma^+)*\dot{\S}_\gamma$ at inaccessible stages $\gamma<\kappa$, and $\Q_\kappa=\Add(\kappa,\kappa^{++})$ at stage $\kappa$. Suppose that $V[G][g]$ is the corresponding extension. The argument of theorem \ref{Theorem.LeastWC=LeastWM} shows that $\kappa$ is the least weakly compact cardinal and weakly measurable in $V[G][g]$. Let us argue also that it is unfoldable there. Fix any $A\of\kappa$ in $V[G][g]$ and any large beth fixed point $\theta=\beth_\theta$. Once again, we may assume $A\in V[G][g_0]$, where $g_0$ is the first Cohen set added by $g$ at stage $\kappa$. Thus, there is a name $\dot A\in V$ of hereditary size $\kappa$ such that $A=\dot A_{G*g_0}$. Since $\kappa$ is strongly unfoldable, it is also superunfoldable and so there is $j:M\to N$, where $M^\ltkappa\of M\satisfies\ZFC$ and $M\in H_{\kappa^+}$, such that $N^\theta\of N$ and $|N|^V=\beth_{\theta+1}=\theta^+$ and $\theta<j(\kappa)$. As in the proof of theorem \ref{Theorem.LeastWC=LeastUnfoldable}, we may lift $j$ to $j:M[G][g_0]\to N[j(G)][j(g_0)]$, where $j(G)=G*(g\restrict\kappa^+)*T*\Gtail$, where $T$ is obtained via lemma \ref{Lemma.Souslin+Branch=Cohen} from one of the Cohen sets added by $g$ beyond $\kappa^+$. The point is that the difference between adding one Cohen subset to $\kappa$ or $\kappa^+$ many Cohen subsets to $\kappa$ at stage $\kappa$ in $j(\P)$ is not a difference that fundamentally affects the argument of theorem \ref{Theorem.LeastWC=LeastUnfoldable}. The result is that this lifted embedding witnesses the desired instance of unfoldability for $\kappa$, and so $\kappa$ remains unfoldable in $V[G][g]$.
\end{proof}

The hypothesis of a strongly unfoldable measurable cardinal is considerably stronger than the hypothesis merely of a measurable cardinal, since strongly unfoldable cardinals are $\Sigma_2$-reflecting and therefore reflect the existence of the measurable cardinal below. It seems likely to us that the consistency strength of the hypothesis of theorem \ref{Theorem.LeastWC=LeastUnf=LeastWM} can be improved.  Moreover, as was the case in the proof of theorem \ref{Theorem.LeastWC=LeastUnfoldable}, we did not necessarily use all aspects of the strongly unfoldability of $\kappa$.  We used the additional property mentioned after the proof and also the fact that the targets of the embeddings were closed under $\kappa$ sequences to ensure that the stage $\kappa$ forcing of $j(\P)$ was the $\text{Add}(\kappa, \kappa^+)$ of $V[G]$.  In fact, it would have been sufficient to have assumed that $N$ were only closed under $\ltkappa$ sequences because the stage $\kappa$ forcing would then be $\text{Add}(\kappa, \gamma)^{V[G]}$ for some $\gamma \leq \kappa^+$ whereby we could use $(g\restrict\gamma)*T$ for the generic over this forcing instead.

\section{The least weakly compact cardinal can be nearly $\theta$-supercompact}\label{Section.LeastWCisNearlySC}

Next, we prove that the least weakly compact cardinal can be nearly $\theta$-supercompact. A cardinal $\kappa$ is {\df nearly $\theta$-supercompact} if for every $A\of\theta$, there is a transitive set $M \models\ZFC$, with $M^\ltkappa\of M$ and $A, \kappa, \theta\in M$, and another transitive set $N$ with an elementary embedding $j: M \to N$ having critical point $\kappa$, such that $\theta<j(\kappa)$ and $j\image\theta \in N$. The notion was introduced by Schanker in \cite{Schanker2011:Dissertation, Schanker2013:PartialNearSupercompactness}, where he also provided several useful equivalent formulations. For example, when $\theta^\ltkappa=\theta$ one may equivalently assume that $M$ and $N$ have size $\theta$ and are closed under $\ltkappa$-sequences, and furthermore it suffices to have models of $\ZFC^-$ instead of \ZFC; but one should understand $\ZFC^-$, as always, in the manner of \cite{GitmanHamkinsJohnstone:WhatIsTheTheoryZFC-Powerset?}. Another characterization, the normal filter characterization, is that $\kappa$ is nearly $\theta$-supercompact for $\theta=\theta^\ltkappa$ if and only if whenever $\mathcal{A}$ is a collection of $\theta$ many subsets of $P_\kappa\theta$ and $\mathcal{F}$ is a family of $\theta$ many functions from $P_\kappa\theta$ into $\theta$, there is a $\kappa$-complete $\mathcal{F}$-normal fine filter on $P_{\kappa}\theta$, meaning that every $f \in \mathcal{F}$ that is regressive on a set in the filter is constant on a set in the filter.

\begin{theorem}\label{Theorem.LeastWCNearlySC}
 If $\kappa$ is nearly $\theta$-supercompact and $\theta^{{<}\kappa} = \theta$, then there is a forcing extension, preserving all cardinals above $\kappa$, in which $\kappa$ is the least weakly compact cardinal and still nearly $\theta$-supercompact.
\end{theorem}

\begin{proof}
We adapt the method of theorems \ref{Theorem.LeastWC=LeastUnfoldable} and \ref{Theorem.LeastWC=LeastWM}. Suppose $\kappa$ is nearly $\theta$-supercompact and $\theta^{{<}\kappa} = \theta$. Results in \cite{Schanker2011:Dissertation,Schanker2013:PartialNearSupercompactness} show that there is function $f:\kappa\to\kappa$ with the near $\theta$-supercompactness Menas property, which means that one may find a near $\theta$-supercompactness embedding $j:M\to N$ with the additional property that $\theta<j(f)(\kappa)$. We use $f$ to define an iteration $\P=\langle (\P_\gamma,\dot{\Q}_\gamma)\mid\gamma\leq\kappa\rangle$ of length $\kappa+1$ using Easton support as follows. If $\gamma$ is inaccessible in $V[G_\gamma]$, then $\dot\Q_\gamma$ names the two-step iteration $\Add(\gamma,f(\gamma))*\dot\S_\gamma$ that adds $f(\gamma)$ many Cohen subsets to $\gamma$ and then adds a homogeneous $\gamma$-Souslin tree via the forcing of lemma \ref{Lemma.Souslin+Branch=Cohen}, and otherwise names the trivial forcing. At the top of the iteration, at stage $\kappa$, let $\dot{\Q}_\kappa=\dot{\Add}(\kappa,\theta^+)$ be a $\P_\kappa$-name for the forcing to add $\theta^+$ many Cohen subsets to $\kappa$, as defined in $V^{\P_\kappa}$. Let $G*g\of\P_\kappa*\dot{\Q}_\kappa$ be $V$-generic, and consider $V[G][g]$. Since the forcing $\P_\kappa*\dot{\Q}_\kappa$ is $\kappa^+$-c.c., all cardinals above $\kappa$ are preserved. As before, since the forcing at any stage $\gamma<\kappa$ remaining inaccessible in the partial extension $V[G_{\gamma}]$ explicitly adds a $\gamma$-Souslin tree, which survives into $V[G][g]$ since the subsequent forcing is strategically $\leqgamma$-closed, it follows that $V[G][g]$ has no weakly compact cardinals below $\kappa$.

Consider now $\kappa$ itself, which we aim to show is nearly $\theta$-supercompact in $V[G][g]$. In order to do so, suppose in $V[G][g]$ that $\mathcal{A}$ is any collection of $\theta$ many subsets of $(P_{\kappa}\theta)^{V[G][g]}$ and $\mathcal{F}$ is a family of $\theta$ many functions from $(P_{\kappa}\theta)^{V[G][g]}$ into $\theta$. By extending $\mathcal{A}$ if necessary, we may assume that it is a $\ltkappa$-complete Boolean algebra of sets containing all the fineness cones $\set{\sigma\in (P_\kappa\theta)^{V[G][g]} \st \alpha\in \sigma}$ for any $\alpha<\theta$, and furthermore that it contains the preimages of any point under any function in $\mathcal{F}$. Since the forcing $\P_\kappa*\dot{\Q}_\kappa$ is $\kappa^+$-c.c., it follows that the sets $\mathcal{A}$, $\mathcal{F}$ and $(P_\kappa\theta)^{V[G][g]}$, which all have size $\theta$, are in $V[G][g\restrict\delta]$ for some $\delta<\theta^+$. By applying an automorphism to the stage $\kappa$ forcing, we may assume that $\delta=\theta$. Thus, there are names $\dot{\mathcal{A}},\dot{\mathcal{F}},\dot P$ of hereditary size $\theta$ in $V$ whose values by the filter $G*(g\restrict\theta)$ are $\mathcal{A}$, $\mathcal{F}$ and $(P_\kappa\theta)^{V[G][g]}$. Let $j:M\to N$ be a near $\theta$-supercompactness embedding in $V$ with critical point $\kappa$ between ${<}\kappa$-closed \ZFC\ models such that $\dot{\mathcal{A}},\dot{\mathcal{F}},\dot P,f \in M$ with $\theta<j(f)(\kappa)$.  In particular, this means that $N$ is a transitive model of size $\theta$, $j \image \theta \in N$, and $j(\kappa) > \theta$. Since $\kappa$ remains inaccessible in $V[G]$ and hence also in $N[G]$, there will be nontrivial forcing at stage $\kappa$ in $j(\P_\kappa)$, which we may therefore factor as $\P_\kappa*\dot{\Add}(\kappa,\beta)*\dot{\S}_\kappa*\dot{\P}_{\tail}$, where $\beta=j(f)(\kappa)$. Since $N$ has size $\theta$, it follows that $\beta<(\theta^+)^V$. Let $B\of\theta$ be one of the Cohen subsets added by $g$ beyond $\beta$. By lemma \ref{Lemma.Souslin+Branch=Cohen} we may regard $B\iso T*b$ as first adding a $\kappa$-Souslin tree $T$ and then forcing with that tree to add the branch $b$. Thus, the filter $G*(g\restrict\beta)*T$ is $N$-generic for the forcing of $j(\P_\kappa)$ up to and including stage $\kappa$. We now force with $\Ptail=(\dot{\P}_{\tail})_{G*(g\restrict\beta)*T}$ over $V[G][g]$ to add a generic filter $\Gtail\of\Ptail$, which will of course also be $N[G][g\restrict\beta][T]$-generic, and so we may lift the embedding in $V[G][g][\Gtail]$ to $j:M[G]\to N[j(G)]$, where $j(G)=G*(g\restrict\beta)*T*\Gtail$. The Menas property ensures $\theta\leq\beta$, and so $g\restrict\theta$ is also in $N[j(G)]$. Furthermore, using $j\image\theta$ and $g\restrict\theta$ in $N[j(G)]$ we may construct $j\image (g\restrict\theta)\in N[j(G)]$, and this is a condition in $j(\Q_\kappa)$, which is a master condition for the generic filter $g\restrict\theta\of\Add(\kappa,\theta)$. Let us now force with $j(\Add(\kappa,\theta))$ to add a $V[G][g][\Gtail]$-generic filter $h\of j(\Add(\kappa,\theta))$, containing this condition. It follows that $h$ is also $N[j(G)]$-generic and the embedding lifts in $V[G][g][\Gtail][h]$ to $j:M[G][g\restrict\theta]\to N[j(G)][h]$, where $j(g\restrict\theta)=h$. Since the stage $\kappa$ forcing of $j(\P)$ added $\beta$ many subsets to $\kappa$, the next inaccessible cardinal of $N[G][g\restrict\beta][T]$ will be beyond $\theta$, and consequently $\Ptail*j(\Q_\kappa)$ is strategically $\leqtheta$-closed in this model.

Let $\mu$ be the filter on $\mathcal{A}$ generated via the lifted embedding by $j\image\theta$; that is, $X\in\mu\iff X\in \mathcal{A}$ and $j\image\theta\in j(X)$. Since $\mathcal{A},\mathcal{F}\of M[G][g\restrict\theta]$, the usual supercompactness arguments show that $\mu$ is an $\mathcal{F}$-normal fine filter on $P_\kappa\theta$ measuring every set in $\mathcal{A}$. (Technically, $\mu$ is a filter on the Boolean algebra $\mathcal{A}$; one should close it under supersets to have a filter on the power set of $P_\kappa\theta$.) And since $M^\ltkappa\of M$, it follows easily that $M[G][g\restrict\theta]^\ltkappa\of M[G][g\restrict\theta]$ in $V[G][g]$ and even in $V[G][g][\Gtail][h]$, and so $\mu$ is a $\kappa$-complete filter. What remains for us to show is that this filter is actually in $V[G][g]$. For this, we observe first that from a $\theta$-enumeration of $\mathcal{A}$ in $M[G][g\restrict\theta]$, we may use $j\image\theta\in M$ to construct $j\image\mathcal{A}$ and hence also $j\restrict\mathcal{A}$ in $N[j(G)][h]$. From this observation, it follows that $\mu$, as we have defined it, is in $N[j(G)][h]$. But since $\Ptail*j(\Q_\kappa)$ is strategically $\leqtheta$-closed over the stage $\kappa$ extension $N[G][g\restrict\theta][T]$, the set $\mu$ could not have been added by this extra forcing over that model. Thus, $\mu\in N[G][g\restrict\beta][T]$. Since this model is contained in $V[G][g\restrict\beta][T]$, which is contained in $V[G][g]$, it follows that $\mu$ is in $V[G][g]$, and so we have witnessed this instance of near $\theta$-supercompactness in $V[G][g]$, as desired.
\end{proof}

From a stronger hypothesis, we can make the least weakly compact cardinal simultaneously unfoldable, weakly measurable, nearly $\theta$-supercompact and indeed, $\theta^+$-nearly $\theta$-supercompact.  In \cite{Schanker2011:Dissertation, Schanker2013:PartialNearSupercompactness}, Schanker defines a cardinal $\kappa$ to be {\df $\lambda$-nearly $\theta$-supercompact} provided that for every $A \subseteq \kappa^+$, there exists a transitive $M \satisfies \ZFC^-$ for which $\kappa, \theta, A \in M$ and $M^{\ltkappa} \subseteq M$, a transitive $N$, and an elementary embedding $j: M \to N$ with critical point $\kappa$ such that $j(\kappa) > \theta$ and $j \image \theta \in N$.  He then proved equivalent characterizations in the case that $\lambda^{\ltkappa} = \lambda$ and $\lambda \geq \theta$.  In this case, a cardinal $\kappa$ is $\lambda$-nearly $\theta$-supercompact exactly when for every collection $\mathcal{A}$ containing $\lambda$ many subsets of $P_\kappa\theta$ and every family $\mathcal{F}$ containing $\lambda$ many functions from $P_\kappa\theta$ into $\theta$, there is a $\kappa$-complete $\mathcal{F}$-normal fine filter on $P_{\kappa}\theta$.  For example, a cardinal $\kappa$ is weakly measurable if and only if it is $\kappa^+$-nearly $\kappa$-supercompact, and so the $\theta^+$-nearly $\theta$-supercompact cardinals naturally extend this.

\begin{theorem}\label{Theorem.LeastWC=LeastWM=LeastUnf=LeastNearlyThetaSC}
 If $\kappa$ is strongly unfoldable and $\theta$-supercompact, then there is a forcing extension, preserving all cardinals in the interval $[\kappa,\theta]$, in which $\kappa$ becomes the least weakly compact cardinal and is unfoldable, weakly measurable, nearly $\theta$-supercompact and indeed, $\theta^+$-nearly $\theta$-supercompact.
\end{theorem}

\begin{proof}
Suppose that $\kappa$ is strongly unfoldable and $\theta$-supercompact. By replacing $\theta$ with $\theta^\ltkappa$, since $\kappa$ is necessarily also $\theta^\ltkappa$-supercompact, we may assume assume without loss of generality that $\theta^\ltkappa=\theta$. By forcing to add a fast function, if necessary (see \cite{Hamkins2001:UnfoldableCardinals} and \cite{Hamkins2000:LotteryPreparation} for details), we may assume that there is a function $f:\kappa\to\kappa$ which is an ordinal-anticipating Laver function both for strong unfoldability and for $\theta$-supercompactness; that is, we may find embeddings $j$ witnessing these large cardinal properties, for which also $j(f)(\kappa)=\alpha$ for any desired ordinal $\alpha$, in the case of strong unfoldability, and for any desired $\alpha<\theta^+$, in the case of $\theta$-supercompactness. By (class) forcing if necessary, we may assume as in theorem \ref{Theorem.LeastWC=LeastUnfoldable} that there are no inaccessible cardinals above $\theta$  (one may achieve this by first forcing via the lottery preparation \cite{Hamkins2000:LotteryPreparation} to make the $\theta$-supercompactness of $\kappa$ indestructible by $\ltkappa$-directed closed forcing, and then forcing above $\theta$ so as to kill off all inaccessible cardinals there). Similarly, we may also assume that the \GCH\ holds at $\theta$ and above.

Let $\P_\kappa$ be the Easton support forcing iteration of length $\kappa$ used in theorem \ref{Theorem.LeastWCNearlySC}, which forces at stage $\gamma$ with $\Q_\gamma=\Add(\gamma,f(\gamma))*\dot\S_\gamma$, provided that $\gamma$ is inaccessible in $V[G_\gamma]$, and uses trivial forcing otherwise. But this time, at stage $\kappa$ we force with $\Q_\kappa=\Add(\kappa,\theta^{++})$. Suppose that $G*g\of\P_\kappa*\dot\Q_\kappa$ is $V$-generic. As before, there are no weakly compact cardinals below $\kappa$ in $V[G][g]$ since the tree property fails at every inaccessible $\gamma<\kappa$ in this model.

Let us begin by showing at first merely that $\kappa$ remains nearly $\theta$-supercompact in $V[G][g]$. Let $\mathcal{A}$ be a $\theta$-sized family of subsets of $P_\kappa\theta$ and $\mathcal{F}$ be any $\theta$-sized family of functions from $P_\kappa\theta$ to $\theta$. By closing under Boolean operations and $\ltkappa$-intersections, we may assume that $\mathcal{A}$ is a $\ltkappa$-complete Boolean algebra of subsets, containing all the fineness cones and the preimages of points under the functions of $\mathcal{F}$. By the chain condition, it follows that $\mathcal{A},\mathcal{F}\in V[G][g\restrict I]$ for some set $I$ of size at most $\theta$. By applying an automorphism to the forcing to move those coordinates to the front, we may assume without loss of generality that $\mathcal{A},\mathcal{F}\in V[G][g\restrict\theta]$. Let $j:V\to M$ be a $\theta$-supercompactness embedding for $\kappa$, such that $j(f)(\kappa)=\theta$. In particular, $M^\theta\of M$ and $j\image\theta\in M$. The forcing $j(\P_\kappa)$ factors at $\kappa$ as $\P*\dot\Add(\kappa,\theta)*\dot\S_\kappa*\dot{\P}_{\tail}$. Let $B\of\kappa$ be one of the Cohen sets added by $g$ beyond coordinate $\theta$, and by lemma \ref{Lemma.Souslin+Branch=Cohen} regard $B=T*b$, where $T$ is generic for $\S_\kappa$. Thus, $M[G][g\restrict\theta][T]$ is generic for the forcing of $j(\P_\kappa)$ up to and including stage $\kappa$. Furthermore, since we assumed that there are no inaccessible cardinals in $V$ above $\kappa$, it follows that the next nontrivial stage of forcing in $\Ptail=(\dot{\P}_{\tail})_{G*g\restrict\theta*T}$ is after $\theta$, and so $\Ptail$ is strategically $\leqtheta$-closed in $M[G][g\restrict\theta][T]$. Let us simply force to add $\Gtail\of\Ptail$ over $V[G][g]$. We may lift the embedding to $j:V[G]\to M[j(G)]$ in $V[G][g][\Gtail]$, where $j(G)=G*(g\restrict\theta)*T*\Gtail$. Next, we use the fact that $g$ and $j\image\theta$ are  in $M[j(G)]$ in order to form $\bigcup j\image (g\restrict\theta)$, which is a master condition in $j(\Add(\kappa,\theta))$ in the model $M[j(G)]$. Forcing below this condition, we add a $V[G][g][\Gtail]$-generic filter $h\of j(\Add(\kappa,\theta))$, and lift the embedding to $j:V[G][g\restrict\theta]\to M[j(G)][h]$, with $h=j(g\restrict\theta)$, where the lift is a class in $V[G][g][\Gtail][h]$. Let $\mu=\set{X\in \mathcal{A}\st j\image\theta\in j(X)}$. This is easily seen to be an $\mathcal{F}$-normal fine measure on $\mathcal{A}$, which was defined in $V[G][g][\Gtail][h]$. Furthermore, since $\mathcal{A}$ has size $\theta$, it follows that $j\image\mathcal{A}$ and hence also $j\restrict\mathcal{A}\in M[j(G)][h]$. From this, it follows that $\mu\in M[j(G)][h]$. Since the forcing at stages beyond $\kappa$ was strategically $\leqtheta$-closed, it follows that $\mu\in M[G][g\restrict\theta][T]$. Since this is contained within $V[G][g\restrict \theta][T]$, which is contained in $V[G][g]$, we have witnessed the desired instance of near $\theta$-supercompactness in $V[G][g]$, as desired.

We now push the previous argument a bit harder in order to show fully that $\kappa$ is $\theta^+$-nearly $\theta$-supercompact in $V[G][g]$, which will also establish that $\kappa$ is weakly measurable there. For this case, we use a $\theta$-supercompactness embedding $j:V\to M$ for which $j(f)(\kappa)=\theta^+$. Since in this case we assume that $\mathcal{A}$ and $\mathcal{F}$ have size $\theta^+$, we may assume without loss of generality that they are in $V[G][g\restrict\theta^+]$. The \GCH\ at $\theta$ allows us to construct the filter $\Gtail$ by a diagonalization argument inside $V[G][g\restrict\theta^+][T]$, in analogy with the argument of theorem \ref{Theorem.LeastWC=LeastWM}. Although we do not have a master condition for $g\restrict\theta^+$, we do have $j\image(g\restrict\beta)\in M[j(G)]$ for each $\beta<\theta^+$, and this is enough to run the master filter argument as in theorem \ref{Theorem.LeastWC=LeastWM} to find an $M[j(G)]$-generic filter $h\of j(\Add(\kappa,\theta^+)$ with $h\in V[G][g]$ and $j"g\restrict\theta^+\subseteq h$. Thus we may lift the embedding in $V[G][g]$ to $j:V[G][g\restrict\theta^+]\to M[j(G)][h]$, with $j(g\restrict\theta^+)=h$, and this witnesses the desired instance of $\theta^+$-near $\theta$-supercompactness, since the induced measure $\mu=\set{X\of P_\kappa\theta\st X\in V[G][g\restrict\theta^+], j\image\theta\in j(X)}$ is a $\mathcal{F}$-normal fine filter on $P_\kappa\theta$ in $V[G][g]$ measuring every set in $\mathcal{A}$.

Finally, we show that $\kappa$ remains unfoldable in $V[G][g]$. Fix any $A\of\kappa$ in $V[G][g]$. By the chain condition, $A$ is determined by $\kappa$ many coordinates of $g$, and so by applying an automorphism, we may assume that $A\in V[G][g_0]$ where $g_0$ is the first coordinate of $g$. So $A=\dot A_{G*g_0}$. Fix any large beth-fixed point $\lambda=\beth_\lambda$. As in the proof of theorem \ref{Theorem.LeastWC=LeastUnfoldable}, since $\kappa$ is strongly unfoldable it is also superunfoldable, and in fact there is a $j:M\to N$ witnessing the $\lambda$-superunfoldability of $\kappa$, meaning that $M^\ltkappa\of M$, $N^\lambda\of N$, $|N|=\lambda^+$, and $\dot A\in M$. Our assumption on $f$ ensures that we may find such an embedding for which also $j(f)(\kappa)=1$. The forcing $j(\P_\kappa)$ therefore factors as $\P_\kappa*\dot\Add(\kappa,1)*\dot\S_\kappa*\dot{\P}_\tail$. Let $B$ be one of the Cohen subsets of $\kappa$ added after the first coordinate of $g$, and by lemma \ref{Lemma.Souslin+Branch=Cohen} view $B$ as $T*b$, where $T$ is $V[G][g_0]$-generic for $\S_\kappa$. Thus, we may form the partial extension $N[G][g_0][T]$. The remaining forcing $\Ptail=(\dot{\P}_\tail)_{G*g_0*T}$ is strategically $\ltlambda^+$-closed in this model, since there are no inaccessible cardinals above $\kappa$. Since $N[G][g_0][T]^\lambda\of N[G][g_0][T]$ in $V[G][g_0][T]$, we may by diagonalization find an $N[G][g_0][T]$-generic filter $\Gtail\of\Ptail$ and thereby lift the embedding to $j:M[G]\to N[j(G)]$, where $j(G)=G*g_0*T*\Gtail$. Furthermore, using $\bigcup j"g_0\in N[j(G)]$ as a master condition, we may similarly diagonalize to find an $N[j(G)]$ filter $h\of j(\Add(\kappa,1))$ below the condition $g_0$, and thereby lift the embedding to $j:M[G][g_0]\to N[j(G)][h]$, with $j(g_0)=h$. Since $A\in M[G][g_0]$ and $\lambda<j(\kappa)$, we have therefore witnessed the desired instance of unfoldability in $V[G][g]$.
\end{proof}

\section{Making every weakly compact cardinal unfoldable, weakly measurable and nearly $\theta$-supercompact}\label{Section.EveryWC}

In the previous sections, we proved that the least weakly compact cardinal can be unfoldable, weakly measurable, nearly $\theta$-supercompact or more. In this section, we prove a variety of global theorems showing how to ensure that every weakly compact cardinal---a proper class of them---exhibits such extra strength. We shall use hypotheses of various strengths in order to attain the various possible combinations.

\begin{theorem}\label{Theorem.WC=Unf=W}
 Suppose that $W$ is a class of strongly unfoldable cardinals and that $W$ has no inaccessible limit points. Then there is a class forcing extension in which
 the class of weakly compact cardinals and the class of unfoldable cardinals both coincide with $W$. In particular, it is relatively consistent with \ZFC\ that there is a proper class of weakly compact cardinals and each of them is unfoldable.
\end{theorem}

\begin{proof}
Suppose that $W$ is a class of strongly unfoldable cardinals and that $W$ has no inaccessible limit points. By forcing, if necessary, we may suppose without loss of generality that the \GCH\ holds, since the forcing to achieve this preserves all strongly unfoldable cardinals. We shall perform class forcing, in two steps. The first step will kill all weakly compact cardinals: let $\P$ be the Easton support class iteration, which at any stage $\gamma$ that happens to be inaccessible in the partial extension $V[G_\gamma]$, forces with the two-step iteration $\Q_\gamma=\Add(\gamma,1)*\dot\S_\gamma$, which adds a Cohen subset to $\gamma$ and then uses the forcing of lemma \ref{Lemma.Souslin+Branch=Cohen} to add a homogeneous $\gamma$-Souslin tree. Let $G\of\P$ be $V$-generic for this iteration and consider the model $V[G]$. Since the stage $\gamma$ forcing is $\ltgamma$-strategically closed, it follows that the iteration $\P$ is progressively strategically closed, and consequently by general considerations (see \cite{Reitz2006:Dissertation}) we may deduce that $V[G]\satisfies\ZFC$. Furthermore, since the stage $\gamma$ forcing explicitly adds a $\gamma$-Souslin tree and the subsequent stages of forcing do not affect the weak compactness of $\gamma$, it follows that there are no weakly compact cardinals in $V[G]$.

The second step will restore the weak compactness and unfoldability of the cardinals in $W$. Namely, in $V[G]$ let $\R=\Pi_{\kappa\in W}\,(T_\kappa\times\Add(\kappa,\kappa^+))$, using Easton-support product forcing, where $T_\kappa$ is the $\kappa$-Souslin tree added by the stage $\kappa$ forcing of $\P$ (and there was indeed forcing at stage $\kappa$ when $\kappa\in W$). Let $H\of\R$ be $V[G]$-generic, and consider $V[G][H]$, our final model. Consider first an inaccessible cardinal $\gamma$ not in $W$. It is not weakly compact in $V[G]$, and since $\gamma$ is not a limit point of $W$, the subsequent forcing $\R$ factors into the part of the forcing above $\gamma$, which is strategically $\leqgamma$-closed in $V[G]$, and the part of the forcing below $\gamma$, which is small with respect to $\gamma$. Neither factor can make $\gamma$ weakly compact, and so no such $\gamma$ is weakly compact in $V[G][H]$.

It remains to show that if $\kappa\in W$, then $\kappa$ is unfoldable (and hence also weakly compact) in $V[G][H]$. For this, suppose $\kappa\in W$, and consider any $A\of\kappa$ in $V[G][H]$ and any ordinal $\theta$. By increasing $\theta$ if necessary, we may assume that it is beth-fixed point $\theta=\beth_\theta$ and above $\kappa$. We may factor the $\P$ forcing at $\kappa$ as $\P_\kappa*\dot{\Add}(\kappa,1)*\dot{\S}_\kappa*\dot{\P}^\kappa$, with the corresponding factorization of the generic filter as $G_\kappa*g_\kappa*T_\kappa*G^\kappa$. We may similarly factor the $\R$ forcing as $\R_\kappa\times T_\kappa\times\Add(\kappa,\kappa^+)\times\R^\kappa$, where $\R_\kappa$ is the part of $\R$ below $\kappa$ and $\R^\kappa$ is the part above $\kappa$. Let us denote the corresponding factorization of the generic filter $H$ as $H_\kappa\times b_\kappa\times h_\kappa\times H^\kappa$, where $b_\kappa$ is the branch added by forcing with the tree $T_\kappa$ and $h_\kappa\of\Add(\kappa,\kappa^+)$ are the subsequent Cohen sets added by the rest of the stage $\kappa$ factor of $\R$, and where $H_\kappa\of\R_\kappa$ and $H^\kappa\of\R^\kappa$ are the parts of $H$ at the factors below and above $\kappa$, respectively. Since $A\of\kappa$ is in $V[G][H]$, it follows that $A\in V[G_\kappa][g_\kappa][T_\kappa][H_\kappa][b_\kappa][h_\kappa]$, as the forcing above $\kappa$ in both $\P$ and $\R$ is strategically $\leqkappa$-closed. Furthermore, since $\Add(\kappa,\kappa^+)$ is $\kappa^+$-c.c., it follows that we do not need all of the $h_\kappa$ forcing to construct $A$, and so $A\in V[G_\kappa][g_\kappa][T_\kappa][H_\kappa][b_\kappa][h_\kappa\restrict\beta]$ for some $\beta<\kappa^+$, where $h_\kappa\restrict\beta$ is the restriction of $h_\kappa$ to $\Add(\kappa,\beta)$. Indeed, by applying an automorphism of the forcing, we may assume that $A\in V[G_\kappa][g_\kappa][T_\kappa][H_\kappa][b_\kappa][h_\kappa^0]$, where $h_\kappa^0$ is the very first Cohen subset of $\kappa$ added by $h_\kappa$. Since the $\R$ forcing is a product, we may rearrange the factors as $V[G_\kappa][g_\kappa][T_\kappa][b_\kappa][h_\kappa^0][H_\kappa]$. Lemma \ref{Lemma.Souslin+Branch=Cohen} shows that the forcing $T_\kappa*b_\kappa$, which adds the tree and then forces with it, is the same as adding a single Cohen subset of $\kappa$. In particular, there is a name $\dot A\in V$ of hereditary size $\kappa$ in $V$ which has value $A$ with respect to the generic filter we have just described for this extension.

Since $\kappa$ is strongly unfoldable in $V$, there is a $\theta$-superunfoldability embedding $j:M\to N$, where $M\satisfies\ZFC$ is a transitive set of size $\kappa$ with $M^\ltkappa\of M$ and $\dot A\in M$, and where $N$ is a transitive set with $N^\theta\of N$ and $|N|^V=\beth_{\theta+1}=\theta^+$, where $\theta<j(\kappa)$. Without loss of generality, by increasing $\theta$ if necessary, we may assume that $\theta$ is a singular beth-fixed point which is not a limit of inaccessible cardinals. Let us lift $j$ through $G_\kappa$. Since $N^\theta\subseteq N$ it follows that the forcing $j(\P_\kappa)$ factors as $\P_{\theta}*\dot{\P}_{\tail}$ where $\dot{\P}_{\tail}$ is a $\P_{\theta}$-name for the tail of the iteration $j(\P_\kappa)$ from stage $\theta$ to stage $j(\kappa)$. Since $\theta$ is not a limit of inaccessible cardinals, it follows that on a tail the iteration $\P_\theta$ is trivial. This implies that $\P_\theta$ is $\theta^+$-c.c. and thus $N[G_\theta]^\theta\subseteq N[G_\theta]$ in $V[G_\theta]$. Since $\theta$ is singular in $N[G_\theta]$ it follows that $\P_{\tail}=(\dot{\P}_{\tail})_{G_{\theta}}$ is strategically ${\leq}\theta$-closed in $N[G_\theta]$. Thus, as in the proof of theorem 4, we may build an $N[G_{\theta}]$-generic filter $G_{\tail}$ for $\P_{\tail}$ in $V[G_{\theta}]$. This implies that $j$ lifts in $V[G_{\theta}]$ to $j:M[G_\kappa]\to N[j(G_\kappa)]$ where $j(G_\kappa)=G_{\theta}*G_{\tail}$. Using the fact that $T_\kappa*b_\kappa*h_\kappa^0$ can be viewed as a single Cohen subset of $\kappa$, the master condition argument given in theorem \ref{Theorem.LeastWC=LeastUnfoldable} shows that this embedding lifts further to $j:M[G_\kappa][T_\kappa][b_\kappa][h_\kappa^0]\to N[j(G_\kappa)][j(T_\kappa)][j(b_\kappa)][j(h_\kappa^0)]$ in $V[G][H]$. Furthermore, since $H_\kappa\of\R_\kappa$ is small relative to $\kappa$, we may also easily lift the embedding to $j:M[G_\kappa][T_\kappa][b_\kappa][h_\kappa^0][H_\kappa]\to N[j(G_\kappa)][j(T_\kappa)][j(b_\kappa)][j(h_\kappa^0)][H_\kappa]$. Since $A$ is now an element of the domain of this embedding, we have thereby witnessed this instance of unfoldability, and so $\kappa$ is unfoldable in $V[G][H]$, as desired.

Finally, the last sentence in the statement of the theorem follows from what we have already proved, by starting with any model having a proper class of unfoldable cardinals, going to $L$ where these cardinals are strongly unfoldable, chopping off the model if this class should have an inaccessible limit point, and then producing a class forcing extension where these cardinals become the weakly compact cardinals and remain unfoldable.
\end{proof}

Next, we arrange that every weakly compact cardinal is weakly measurable.

%
\begin{theorem}\label{Theorem.WC=WM=W}
 Suppose that $W$ is any class of measurable cardinals. Then there is a class forcing extension $V[G]$ in which the class of weakly compact cardinals and the class of weakly measurable cardinals both coincide with $W$.
\end{theorem}

\begin{proof} We follow the proof of theorem \ref{Theorem.LeastWC=LeastWM}, continuing the iteration further so as to handle all the cardinals in $W$. Suppose that $W$ is any class of measurable cardinals in $V$. By forcing if necessary, we may assume that the \GCH\ holds in $V$, since the canonical forcing to achieve this preserves all measurable cardinals. Let $\P$ be the Easton-support class iteration, which forces only at inaccessible cardinal stages $\gamma$ in such a way that if $\gamma\notin W$, then the stage $\gamma$ forcing is $\Q_\gamma=\Add(\gamma,\gamma^+)*\dot\S_\gamma$ as defined in $V[G_\gamma]$, and otherwise, if $\gamma\in W$, then the stage $\gamma$ forcing is $\Q_\gamma=\Add(\gamma,\gamma^{++})$ as defined in $V[G_\gamma]$. Suppose that $G\of\P$ is $V$-generic for this forcing, and consider the extension $V[G]$, the desired final model. It is clear that if $\gamma\notin W$, then $\gamma$ cannot be weakly compact in $V[G]$, since in this case the stage $\gamma$ forcing added a $\gamma$-Souslin tree, which ruins the tree property for $\gamma$, and this tree is not disturbed by the subsequent forcing after stage $\gamma$. So we must only prove that every $\kappa\in W$ remains weakly measurable in $V[G]$ (and hence also weakly compact there). Consider any such $\kappa\in W$. Since the forcing beyond stage $\kappa$ is strategically closed well beyond $\kappa$, it suffices to prove that $\kappa$ is weakly measurable in $V[G_\kappa][g_\kappa]$, where $g_\kappa\of\Add(\kappa,\kappa^{++})$ is the generic filter added at stage $\kappa$. For this, suppose that $A\of P(\kappa)$ is a family of subsets of $\kappa$, with $A$ having size at most $\kappa^+$ in $V[G][g_\kappa]$. By the chain condition, it follows that $A\in V[G][g_\kappa\restrict\delta]$ for some $\delta<\kappa^{++}$, and by applying an automorphism to the forcing, we may assume without loss that $\delta=\kappa^+$. Fix in $V$ an elementary embedding $j:V\to M$, the ultrapower by a normal measure on $\kappa$. Consider the forcing $j(\P_\kappa)$, which is the corresponding forcing iteration in $M$ up to stage $j(\kappa)$. In $M$ we may factor this iteration as $j(\P_\kappa)=\P_\kappa*\Qdot_\kappa^M*\Ptail$, where $\Qdot_\kappa^M$ is either $\Add(\kappa,\kappa^+)*\dot\S_\kappa$ or $\Add(\kappa,\kappa^{++})$, as defined in $M[G_\kappa]$, depending on whether or not $\kappa\in j(W)$. In the case $\kappa\notin j(W)$, we can use $h_\kappa=g_\kappa*T_\kappa$, where $T_\kappa$ is extracted from a coordinate in $g_\kappa$ beyond $\kappa^+$ as in the proof of theorem \ref{Theorem.LeastWC=LeastWM}; in the case $\kappa\in j(W)$, then we may simply use $h_\kappa=g_\kappa\restrict(\kappa^{++})^M\of\Q_\kappa^M$. (Note that $(\kappa^{++})^{M[G_\kappa]}=(\kappa^{++})^M<j(\kappa)<\kappa^{++}$, in light of the \GCH\ in $V$.) Since $M^\kappa\of M$ in $V$, it follows that $M[G_\kappa][h_\kappa]^\kappa\of M[G_\kappa][h_\kappa]$ in $V[G_\kappa][h_\kappa]$, and so we may by the usual diagonalization construct an $M[G_\kappa][h_\kappa]$-generic filter $\Gtail\of\Ptail$, and thus lift the embedding $j$ to $j:V[G_\kappa]\to M[j(G_\kappa)]$, where $j(G_\kappa)=G_\kappa*h_\kappa*\Ptail$. We may now use the master-filter argument as in the proof of theorem \ref{Theorem.LeastWC=LeastWM} to lift the embedding to $j:V[G_\kappa][g_\kappa\restrict\kappa^+]\to M[j(G_\kappa)][j(g_\kappa\restrict\kappa^+)]$ in $V[G]$. Since $A\in V[G_\kappa][g_\kappa\restrict\kappa^+]$, this lifted embedding witnesses the desired instance of weak measurability of $\kappa$, and so $\kappa$ is weakly measurable in $V[G]$, as desired.
\end{proof}

One may also undertake a two-step proof of theorem \ref{Theorem.WC=WM=W} in the style of theorem \ref{Theorem.WC=Unf=W}, when $W$ has no inaccessible limit points. Combining this with the proof of theorem \ref{Theorem.WC=Unf=W} establishes:

\begin{theorem}
 If $W$ is any class of strongly unfoldable measurable cardinals and $W$ has no inaccessible limit points, then there is a forcing extension in which the class of weakly compact cardinals, the class of unfoldable cardinals and the class of weakly measurable cardinals each coincides with $W$.
\end{theorem}

The following theorem shows the general flexibility of the situation.

\begin{theorem}\label{Theorem.WCsInWAllNearlySC}
Suppose that $W$ is a class of cardinals and $\kappa\mapsto\theta_\kappa$ is a function such that
\begin{enumerate}
 \item Every $\kappa\in W$ is nearly $\theta_\kappa$-supercompact and $\theta_{\kappa}^{{<}\kappa} = \theta_{\kappa}$.
 \item If $\kappa\in W$, then the next element of $W$ above $\kappa$ is above $\theta_\kappa$.
 \item There exists a class function $f: \text{ORD} \to \text{ORD}$ such that for every $\kappa \in W$,
 $f\image\kappa \subseteq \kappa$, and for every $A \subseteq \theta_{\kappa}$, there is a near
 $\theta_{\kappa}$-supercompactness embedding $j: M \rightarrow N$ with critical point $\kappa$ between
 $\ltkappa$-closed transitive models of $\text{ZFC}^{-}$ of size $\theta_\kappa$ such that $A, f\restrict\kappa \in M$ and
 $j(f\restrict\kappa)(\kappa) > \theta_\kappa$.
\end{enumerate}
Then there is a class forcing extension in which $W$ becomes exactly the class of weakly compact cardinals, and every $\kappa\in W$ remains nearly $\theta_{\kappa}$-supercompact, while furthermore, all cardinals in $[\kappa, \theta_{\kappa}]$ are preserved.
\end{theorem}

\noindent Note that condition (3) follows if $W$ is scattered, that is, if it contains none of its limit points, and this requirement can be viewed as the existence of a sort of $W$-universal Menas function for various degrees of partial near supercompactness.  To construct such a function $f$ when $W$ is scattered, one can define $f(\gamma)$ to be $f_{\kappa}(\gamma)$ where $\kappa$ is the least cardinal in $W$ above $\gamma$ and $f_{\kappa}$ is the Menas function for the near $\theta_{\kappa}$-supercompactness of $\kappa$ defined in \cite{Schanker2011:Dissertation,Schanker2013:PartialNearSupercompactness}. More generally, the property can hold even when $W$ is not scattered, provided there is a sufficient weak reflectivity in the function $\kappa\mapsto\theta_\kappa$.

\begin{proof}
We iteratively apply the argument of theorem \ref{Theorem.LeastWCNearlySC} with a few modifications.  Let $f:\Ord\to\Ord$ be the function whose existence was posited by condition (3), and then let $\P$ be the Easton support class iteration defined as follows. At stage $\gamma$, if $\gamma\notin W$ and $\gamma$ is inaccessible in the partial extension after forcing with $\P_\gamma$, then the stage $\gamma$ forcing is $\Q_\gamma=\Add(\gamma,f(\gamma))*\dot{\S}_\gamma$. At a stage $\gamma$, if $\gamma\in W$, then the stage $\gamma$ forcing is is $\Q_\gamma=\Add(\gamma,\theta_{\gamma}^+)$.  Suppose that $G\of\P$ is $V$-generic, and consider the corresponding extension $V[G]$. Since this is progressively closed forcing in the sense of \cite{Reitz2006:Dissertation}, we have $V[G]\satisfies\ZFC$. Note that each inaccessible cardinal $\gamma$ of $V$ that's not in $W$ and remains inaccessible in $V[G]$ gains a homogeneous $\gamma$-Souslin tree in $V[G]$ and consequently is not weakly compact there.  Therefore, the only remaining weakly compact cardinals in $V[G]$ must be in $W$.  So let's fix any cardinal $\kappa \in W$ and show that it remains nearly $\theta_{\kappa}$-supercompact. Since the stage $\kappa$ forcing ensures $2^\kappa\geq\theta_\kappa^+$, there are no inaccessible cardinals in $V[G_{\kappa+1}]$ in the interval $(\kappa,\theta_\kappa]$. Consequently, since also $W\intersect(\kappa,\theta_\kappa]=\emptyset$, there are no nontrivial stages of forcing of $\P$ in that interval. Consequently, the forcing beyond stage $\kappa$ will be strategically $\leqtheta_\kappa$-closed. Thus, it suffices to show that $\kappa$ remains nearly $\theta_{\kappa}$-supercompact in $V[G_{\kappa+1}]$. For this, we appeal essentially to the lifting argument from theorem \ref{Theorem.LeastWCNearlySC}, using $\theta=\theta_\kappa$.  As in the proof of that theorem, we fix any family $\mathcal{A}$ of at most $\theta$ many subsets of $(P_\kappa\theta)^{V[G]}$ and a family $\mathcal{F}$ of at most $\theta$ many functions from $(P_\kappa\theta)^{V[G]}$ to $\theta$ and then extend $\mathcal{A}$ if necessary to obtain a $\ltkappa$-complete Boolean algebra of sets containing all the fineness cones and preimages of any point under any function in $\mathcal{F}$.  We then take names $\dot{\mathcal{A}},\dot{\mathcal{F}},\dot P \in M$ having hereditary size $\theta$ in $V$ for $\mathcal{A}$, $\mathcal{F}$, and $(P_{\kappa}\theta)^{V[G]}$, respectively, and code it along with $f_{\kappa} = f\restrict\kappa$ as a single subset of $\theta$.  By condition (3), we may find a near $\theta$-supercompactness embedding $j:M\to N$ with critical point $\kappa$ between transitive $\text{ZFC}^-$ models of size $\theta$ such that the coding subset is in $M$ and $j(f_{\kappa})(\kappa) > \theta$.  Our hypotheses ensure that $\theta_\delta<\kappa$ for every $\delta\in W$ below $\kappa$, and so the initial segments of the forcing $\P_{\gamma}$ for $\gamma<\kappa$ are in $V_\kappa$. Therefore, because $N$ has size $\theta$ in $V$, we may factor $j(\P)$ as $\P*\Add(\kappa,\beta)*\Ptail$ or $\P*\Add(\kappa,\beta)*\S_\kappa*\Ptail$ for some $\beta \in [\theta, (\theta^+)^{V})$, depending on whether or not $\kappa \in j(W)$, where $\Ptail$ is the forcing beyond stage $\kappa$.  With this, we can now proceed as in the proof of theorem \ref{Theorem.LeastWCNearlySC} to prove that $\kappa$ remains nearly $\theta$-supercompact in $V[G_{\kappa+1}]$, which is an extension analogous to the $V[G][g]$ from that theorem.  Finally, because $\P_{\kappa+1}$ is $\kappa^+$-c.c. and the remaining forcing is strategically $\leqtheta_\kappa$-closed, no cardinals from $\kappa$ to $\theta_{\kappa}$ are collapsed.
\end{proof}

\begin{corollary}\label{Corollary.AltIdentityCrisis}
 There is a forcing extension of the universe, preserving all nearly $\kappa^+$-supercompact cardinals $\kappa$, in which the class of nearly $\kappa^+$-supercompact cardinals $\kappa$ coincides with the class of weakly compact cardinals.  In particular, in this model the class of weakly measurable cardinals also coincides with the class of weakly compact cardinals.
\end{corollary}

\begin{proof}
 Simply let $W$ be the class of nearly $\kappa^+$-supercompact cardinals $\kappa$ and let $\theta_\kappa=\kappa^+$. Requirements (1) and (2) of theorem \ref{Theorem.WCsInWAllNearlySC} are immediate.  For requirement (3), define $f$ by $\alpha \mapsto \alpha^{+}+1$.  Then for every $A\subseteq\kappa^+$, we can find a near $\kappa^+$-supercompactness embedding $j: M\to N$ between $\ltkappa$-closed transitive models of $\text{ZFC}^{-}$ of size $\theta$ with critical point $\kappa$ where $A \in M$ and $j(f)(\kappa) > \kappa^+$ by making sure our $M$ has a bijection from $|\gamma|$ onto $\gamma$ for each $\nobreak{\gamma < \kappa^+}$ by using a subset of $\theta$ also coding these bijections.
\end{proof}

A similar argument applies to the case of near $\kappa^{++}$-supercompactness and so on. Extending this, let us
say that a cardinal $\kappa$ is nearly supercompact to a weakly inaccessible degree, if it is nearly
$\theta$-supercompact for some weakly inaccessible cardinal $\theta > \kappa$ for which $\theta^{\ltkappa} =
\theta$.  Then we may introduce the following corollary.

\begin{corollary}
 There is a forcing extension of the universe, preserving every cardinal that is nearly supercompact to a weakly inaccessible degree, in which the class of cardinals that are nearly supercompact to a weakly inaccessible degree coincides with the class of weakly compact cardinals.
\end{corollary}

\begin{proof}
Let $W$ be the class of cardinals $\kappa$ that are nearly supercompact to a weakly inaccessible degree, and let $\theta_\kappa$ be the least weakly inaccessible cardinal above $\kappa$ for which $\theta_\kappa^{\ltkappa}=\theta_\kappa$; thus, $\kappa$ is nearly $\theta_\kappa$-supercompact and we fulfill condition (1) of theorem \ref{Theorem.WCsAllNearlySC}. Condition (2) is satisfied by the fact that every $\kappa\in W$ is inaccessible and a limit of inaccessible cardinals. For condition (3), define $f$ by the map that takes every ordinal $\gamma$ to one more than the least (strongly) inaccessible cardinal above it.  Code a bijection from $|\gamma|$ onto $\gamma$ for each $\nobreak{\gamma < \theta_{\kappa}}$ with our desired subset of $\theta_\kappa$ as a single $A \of\theta_\kappa$. Then we can find a near $\theta_{\kappa}$-supercompactness embedding $j: M\to N$ between $\ltkappa$-closed transitive models of $\text{ZFC}^{-}$ of size $\theta_\kappa$ with critical point $\kappa$ where $A \in M$.  Then by use of $j\image\theta_{\kappa}$, $N$ knows that $\theta_{\kappa}$ is the least weakly inaccessible cardinal above $\kappa$ such that $\theta_{\kappa}^{\ltkappa} = \theta_{\kappa}$ and so $j(f)(\kappa) > \theta_\kappa$, verifying condition (3).  By theorem \ref{Theorem.WCsAllNearlySC}, therefore, there is a forcing extension in which $W$ becomes the class of weakly compact cardinals, and every $\kappa\in W$ remains nearly $\theta_\kappa$-supercompact. Furthermore, an analysis of the forcing shows that each $\theta_\kappa$ remains weakly inaccessible in the extension, and so in that extension every weakly compact cardinal is nearly supercompact to a weakly inaccessible degree, as desired.
\end{proof}

The following corollary illustrates the extremely flexible nature of the argument. The hypothesis here is stronger than necessary, in light of theorem \ref{Theorem.WCsInWAllNearlySC}, but this allows for a simple statement of the result.

\begin{corollary}\label{Theorem.WCsAllNearlySC}
Suppose that there is a proper class of supercompact cardinals and $\kappa\mapsto\theta_\kappa$ is any proper class function. Then there is a forcing extension with a proper class of weakly compact cardinals $\kappa$, each of them nearly $\theta_\kappa$-supercompact.  Moreover, the forcing preserves all cardinals in the interval $[\kappa,\theta_\kappa]$, for any weakly compact cardinal $\kappa$ of the extension.
\end{corollary}

\begin{proof}
Fix any class function $\kappa\mapsto\theta_\kappa$. We shall apply theorem \ref{Theorem.WCsInWAllNearlySC}, but using the slightly revised function $\kappa\mapsto\theta_\kappa^{\ltkappa}$. We define the class $W$ by recursion, placing a supercompact cardinal $\kappa$ into $W$ if it is strictly bigger than the supremum of $\theta_\delta$ for all $\delta<\kappa$ previously placed into $W$. Thus, $W$ is scattered proper class of supercompact cardinals, and we have fulfilled the conditions of theorem \ref{Theorem.WCsInWAllNearlySC}.
\end{proof}

We can allow that $W$ contains some of its limit points, when the function $\kappa\mapsto\theta_\kappa$ is sufficiently reflective.

\bibliographystyle{alpha}
\bibliography{HamkinsBiblio,MathBiblio}

\end{document}

%% file: MathMacrosJDH.tex
\newtheorem{theorem}{Theorem}
\newtheorem{maintheorem}[theorem]{Main Theorem}
\newtheorem{corollary}[theorem]{Corollary}

\newtheorem{lemma}[theorem]{Lemma}

\newcommand{\QED}{\end{proof}}

\def\proclaim[#1]{{\bf #1}}
\def\BF#1.{{\bf #1.}}
\newcommand{\url}[1]{{\tt #1}}

\renewcommand{\P}{{\mathbb P}}
\newcommand{\Q}{{\mathbb Q}}

\newcommand{\R}{{\mathbb R}}

\newcommand{\Gtail}{{G_{\!\scriptscriptstyle\rm tail}}}

\newcommand{\Ptail}{{\P_{\!\scriptscriptstyle\rm tail}}}

\newcommand{\Adot}{{\dot A}}

\newcommand{\Qdot}{{\dot\Q}}

\newcommand{\of}{\subseteq}

\newcommand{\set}[1]{\{\,{#1}\,\}}

\newcommand{\Add}{\mathop{\rm Add}}
\newcommand{\Aut}{\mathop{\rm Aut}}

\newcommand{\image}{\mathbin{\hbox{\tt\char'42}}}

\newcommand{\restrict}{\upharpoonright} 
\newcommand{\satisfies}{\models}

\newcommand{\intersect}{\cap}

\newcommand{\smalllt}{\mathrel{\mathchoice{\raise2pt\hbox{$\scriptstyle<$}}{\raise1pt\hbox{$\scriptstyle<$}}{\raise0pt\hbox{$\scriptscriptstyle<$}}{\scriptscriptstyle<}}}
\newcommand{\smallleq}{\mathrel{\mathchoice{\raise2pt\hbox{$\scriptstyle\leq$}}{\raise1pt\hbox{$\scriptstyle\leq$}}{\raise1pt\hbox{$\scriptscriptstyle\leq$}}{\scriptscriptstyle\leq}}}

\newcommand{\ltkappa}{{{\smalllt}\kappa}}
\newcommand{\leqkappa}{{{\smallleq}\kappa}}

\newcommand{\leqgamma}{{{\smallleq}\gamma}}

\newcommand{\ltlambda}{{{\smalllt}\lambda}}
\newcommand{\ltgamma}{{{\smalllt}\gamma}}

\newcommand{\leqtheta}{{{\smallleq}\theta}}

\newcommand{\boolval}[1]{\mathopen{\lbrack\!\lbrack}\,#1\,\mathclose{\rbrack\!\rbrack}}
\def\[#1]{\boolval{#1}}
\newbox\gnBoxA
\newdimen\gnCornerHgt
\setbox\gnBoxA=\hbox{\tiny$\ulcorner$}
\global\gnCornerHgt=\ht\gnBoxA
\newdimen\gnArgHgt
\def\gcode #1{%
\setbox\gnBoxA=\hbox{$#1$}%
\gnArgHgt=\ht\gnBoxA%
\ifnum     \gnArgHgt<\gnCornerHgt \gnArgHgt=0pt%
\else \advance \gnArgHgt by -\gnCornerHgt%
\fi \raise\gnArgHgt\hbox{\tiny$\ulcorner$} \box\gnBoxA %
\raise\gnArgHgt\hbox{\tiny$\urcorner$}}
\newcommand{\UnderTilde}[1]{{\setbox1=\hbox{$#1$}\baselineskip=0pt\vtop{\hbox{$#1$}\hbox to\wd1{\hfil$\sim$\hfil}}}{}}
\newcommand{\Undertilde}[1]{{\setbox1=\hbox{$#1$}\baselineskip=0pt\vtop{\hbox{$#1$}\hbox to\wd1{\hfil$\scriptstyle\sim$\hfil}}}{}}
\newcommand{\undertilde}[1]{{\setbox1=\hbox{$#1$}\baselineskip=0pt\vtop{\hbox{$#1$}\hbox to\wd1{\hfil$\scriptscriptstyle\sim$\hfil}}}{}}
\newcommand{\UnderdTilde}[1]{{\setbox1=\hbox{$#1$}\baselineskip=0pt\vtop{\hbox{$#1$}\hbox to\wd1{\hfil$\approx$\hfil}}}{}}
\newcommand{\Underdtilde}[1]{{\setbox1=\hbox{$#1$}\baselineskip=0pt\vtop{\hbox{$#1$}\hbox to\wd1{\hfil\scriptsize$\approx$\hfil}}}{}}

\newcommand{\st}{\mid}
\renewcommand{\th}{{\hbox{\scriptsize th}}}

\renewcommand{\iff}{\mathrel{\leftrightarrow}}

\newcommand{\iso}{\cong}
\def\<#1>{\langle#1\rangle}

\newcommand{\cp}{\mathop{\rm cp}}

\newcommand{\Ord}{\mathop{{\rm Ord}}}

\newcommand{\ZFC}{{\rm ZFC}}

\newcommand{\GCH}{{\rm GCH}}

\newcommand{\cell}[1]{\boxit{\hbox to 17pt{\strut\hfil$#1$\hfil}}}
\newcommand{\head}[2]{\lower2pt\vbox{\hbox{\strut\footnotesize\it\hskip3pt#2}\boxit{\cell#1}}}
\newcommand{\boxit}[1]{\setbox4=\hbox{\kern2pt#1\kern2pt}\hbox{\vrule\vbox{\hrule\kern2pt\box4\kern2pt\hrule}\vrule}}
\newcommand{\Col}[3]{\hbox{\vbox{\baselineskip=0pt\parskip=0pt\cell#1\cell#2\cell#3}}}
\newcommand{\tapenames}{\raise 5pt\vbox to .7in{\hbox to .8in{\it\hfill input: \strut}\vfill\hbox to
.8in{\it\hfill scratch: \strut}\vfill\hbox to .8in{\it\hfill output: \strut}}}
\newcommand{\Head}[4]{\lower2pt\vbox{\hbox to25pt{\strut\footnotesize\it\hfill#4\hfill}\boxit{\Col#1#2#3}}}
\newcommand{\Dots}{\raise 5pt\vbox to .7in{\hbox{\ $\cdots$\strut}\vfill\hbox{\ $\cdots$\strut}\vfill\hbox{\
$\cdots$\strut}}}
\newcommand{\df}{\it} 